\documentclass[]{scrartcl}

\usepackage{amssymb}
\usepackage[english]{babel}
\usepackage{amsmath}
\usepackage{amsthm}
\usepackage{showlabels}
\usepackage{xcolor}
\usepackage{pgfplots}
\usepackage{pgfplotstable} 
\usepackage{tikz}
\usepackage{algorithm}
\usepackage{algpseudocode}
\usetikzlibrary{spy}

\renewcommand\div{\operatorname{\mathrm{div}}}

\newcommand{\dd}{\mathrm{d}}
\newcommand{\korn}{C_{\mathrm{k}}}
\newcommand{\bmax}{\overline{\beta}}
\newcommand{\bmin}{\underline{\beta}}
\newcommand{\etamin}{\underline{\eta}}
\newcommand{\etamax}{\overline{\eta}}

\newtheorem{theorem}{Theorem}[section]
\newtheorem{lemma}[theorem]{Lemma}

\theoremstyle{definition}
\newtheorem{definition}[theorem]{Definition}

\newtheorem{corollary}[theorem]{Corollary}
\theoremstyle{remark}

\newcommand{\as}{\text{as}}
\title{Least-Squares finite element method for the simulation of sea-ice motion}
\author{Fleurianne Bertrand and Henrik Schneider}

\begin{document}

\maketitle
\section*{Introduction}

        A nonlinear sea-ice problem is considered in a least-squares finite element setting. The corresponding variational formulation approximating simultaneously the stress tensor and the velocity is analysed. In particular,
        the least-squares functional is coercive and continuous in an appropriate solution space and this proves the well-posedness of the problem. As the method does not require a compatibility condition between the finite element space, the formulation allows the use of piecewise polynomial spaces of the same approximation order for both the stress and the velocity approximations. A Newton-type iterative method is used to linearize the problem and numerical tests are provided to illustrate the theory.

\section{Introduction}
Ice and snow-covered surfaces reflect more than half of the solar radiation they are receiving. 
On earth, sea ice represents on average 70\% of the ice and snow-covered surfaces and plays 
therefore a major role in climate modelling.  A crucial component of most sea ice models is
the viscous-plastic rheology (VP) after Hibler  \cite{Hibler1979}.
Recently two papers \cite{Liu2022, Brandt2022} gave answers for locally strong solutions under strict assumptions,
which are the first works on this topic.
The question of the existence of weak solutions remains still open today.
The development of new numerical methods to deal with sea ice behaviour has long been of keen interest, e.g. the EVP model \cite{Hunke1997,Hunke2001} or implicit VP \cite{Hutchings2004, Lemieux2008}. 
Recently the first error estimator for sea-ice simulation was introduced in \cite{Mehlmann2020}.
The authors propose a dual-weighted residual estimator, but the work lacks theoretical results
especially the reliability and efficiency of the estimator.

Sea ice is a complex material and the approximation of the ice stress is a crucial quantity in multiphysics climate simulations.  
In order to avoid the lack of accuracy occurring while computing the ice stresses from the velocity in a post-processing step, 
this paper develops a variational formulation involving the stresses as an
independent $H(\div)$-conforming variable.  To this purpose, an equivalent first-order system of partial differential equations is derived and their residuals are minimized in a least-squares sense.  Moreover, the least-squares functional provides an inherent error estimator to be used in an adaptive refinement process.
In order to deal with the non-linearity,  Gauss-Newton methods can be used to solve the algebraic least-squares problems,  see \cite{S20}. Nonlinear multilevel methods for this type of least squares computations are studied also in \cite{korsawe2004multilevel}. The methode was also before utiliesed 
to compute problems in glaciology \cite{Gunzburger16}.

In Section \ref{sec:model}, the general coupled sea-ice model and useful notations and definitions are reviewed.
In Section \ref{sec:LS}, we introduce a decoupled, time-discrete and semi-linear problem. 
Also, the least-square minimization problem and the necessary tools will be provided. 
In Section \ref{sec:station}, the stationary case is investigated and the coercivity of the least-squares 
functional will be proven under the assumption that the sea-water and sea-ice velocities are close enough 
to each other.
In Section \ref{sec:time}, we prove that the least-squares functional is an efficient and reliable error estimator. 
These results will be generalized to non-symmetric stresses in Section \ref{sec:nonsym}.
Section \ref{sec:experiments} introduces an iterative method based on a Gauss-Newton-type linearization, to solve 
the nonlinear problem. Additionally, a finite element discretization will be presented. The last Section \ref{sec:experiments}
presents some illustrative numerical experiments.

\section{The Sea-Ice Model}\label{sec:model}
According to \cite{Hibler1979},  sea-ice motion on a Lipschitz domain $\Omega \subset \mathbb R^2$
may be modelled  by the equation
\begin{align}
    \rho h u_t = \div \sigma - \tau_o(u) + F(u) \label{eq:momentum}
\end{align}
for the velocity $u: \Omega \rightarrow \mathbb R^2$ and the stresses $\sigma: \Omega \rightarrow \mathbb R^{2\times 2}$
with the ocean current $\tau_o(u) = \rho_oC_o|u-v_o| (u-v_o)$ involving the water velocity $v_o$,
the external forces
\begin{align}
    F(u)   & := \tau_a - \rho h g \nabla H - \rho h f_c k \times u, \label{eq:F}\\
    \tau_a & := \rho_a C_a |v_a| v_a
\end{align}
the Coriolis parameter $f_c$, the gravity  $g$ and the oceanic surface height $H$. 
The euclidian norm is denoted by $|v(x)|=\sqrt{v_1(x)^2+v_2(x)^2}$.
To complete this model,  it remains to relate the stresses $\sigma$ to the strain rate $\varepsilon(u) = \frac{1}{2}\left(\nabla u +(\nabla u)^T\right)$ by the viscous law.
A famous choice here is the viscous-plastic rheology according to \cite{Hibler1979}.  However in this paper and similarly to \cite{Laikhtman1958},  the focus is on the non-linearity of $\tau_o(u)$ and the viscous law
\begin{align}
    \sigma = 2 \eta \varepsilon(u),
\end{align}
where \begin{align} \label{eq:def_p}
    \eta(A,h)=
    \frac{h P^*}{c_m} e^{C(A-1)}
\end{align}
is the shear viscosity, involving the creep $c_m$. 
Moreover, the following constraints hold:
\begin{align}
    0 \leq A \leq 1 ,\qquad h_{\mathrm{min}} \leq h.
\end{align}
This leads to
 \begin{align}
     \frac{h P^*}{c_m} e^{-C}  & \leq \eta(A,h) \leq \frac{h P^*}{c_m} . \label{eq:bbd_P}
\end{align}
Not that the estimate \eqref{eq:bbd_P} for $\eta$ is independent of the sea-ice concentration $A$.
To simplify, we assume that $\eta := \eta(A,h)$ is bounded, in the following way
\begin{align}
\label{eq:bareta}
   {\underline{\eta}} \leq  {\eta} \leq \overline{\eta} \ ,
\end{align}
and give the typical values and ranges for the data in Table \ref{tab:values}. To simplify the notation 
we set $\overline{\eta}=\underline{\eta}^{-1} 
\ .$ As described in \cite[p. 150]{Leppaeranta2011}, natural boundary condition for the problem are given by
\begin{align*}
    \sigma \cdot \mathbf{n} & = 0 \quad \text{open water boundary}, \\
    u \cdot \mathbf{n}      & \leq 0 \quad \text{solid boundary} \ .
\end{align*}
However, in order to focus on the nonlinearity of the problem, they are usually replaced by
\begin{align*}
    u & = 0 \quad \text{solid boundary } \Gamma_D,                               \\
    \sigma \cdot \mathbf{n} &= 0 \quad \text{open water } \Gamma_N.
\end{align*}
Moreover, in order to describe melting and freezing, we let
the sea-ice height $h$ and concentration $A$ be driven by the following problems:
\begin{align*}
    A_t + \div(u A) & = S_A(A, h) \\
    h_t + \div(u h) & = S_h(A, h) \ .
\end{align*}
with the pointwise constraints 
$A(x) \in [0,1]$ and $h(x) \in [0, \infty)$ 
as in \cite{Hibler1979, Semtner1976}. In the rest of this paper, we will therefore assume that $h, A \in L^\infty(\Omega)$.
To enforce this constraint numerically, we will set $h=0$ if $h<h_{\min}$. Note that in this case, the equations simplify to $u=v_o$ and $\sigma =0$.
We will therefore restrict the domain $\Omega$ to the active domain
\begin{align}
    \omega = \{ x \in \Omega : h(x)\geq h_{\min}>0\} \label{eq:active_set}
\end{align}
with the boundary  $\partial\omega = \Gamma_D\cup\Gamma_N$ splitted into $\Gamma_D := \partial \Omega \cap \partial\omega$ and $\Gamma_N := \partial\omega  \setminus \partial\Omega$.

\begin{table*}[t]
    \begin{tabular}{llll} \hline
        Parameter   & Definition                 & Standard value                                       & reasonable Range                              \\ \hline
        $C$         & compaction hardening       & $20$                                                 & $1\ll C \ll \infty$                           \\
        $C_a$       & air drag coefficient       & $1.2\cdot 10^{-3}$                                   & $0.6- 1.2  \cdot 10^{-3}$                     \\
        $C_o$       & water drag coefficient     & $5\cdot 10^{-3}$                                     & $1.6 - 5 \cdot 10^{-3}$                       \\
        $e$         & yield ellipse aspect ratio & $2$                                                  & $1<e\ll \infty$                               \\
        $f_c$       & Coriolis parameter         & $2\varOmega \sin \phi $                              & $\phi$ - latitude                             \\
        $\varOmega$ & rotation rate of the earth & $7.7921 \cdot10^{-5} \mathrm{rad}\ \mathrm{s}^{-1} $ &                                               \\
        $P^*$       & compressive strength       & $25\ \mathrm{kPa}$                                   & $10-100\mathrm{kPa}$                          \\
        $\rho$      & sea ice density            & $900\ \mathrm{kg}\, \mathrm{m}^{-3}$                 & $\approx 900\ \mathrm{kg}\,\mathrm{m}^{-3}$   \\
        $\rho_a$    & air density                & $1.3\ \mathrm{kg}\, \mathrm{m}^{-3}$                 & $\approx 1.3\ \mathrm{kg}\, \mathrm{m}^{-3}$  \\
        $\rho_o$    & seawater density          & $1028\ \mathrm{kg}\, \mathrm{m}^{-3}$                & $\approx 1028\ \mathrm{kg}\, \mathrm{m}^{-3}$ \\
        $v_a$       & air velocity               & $\approx 10\ \mathrm{m}\,\mathrm{s}^{-1}$            & $0-100\ \mathrm{m}\,\mathrm{s}^{-1}$          \\
        $v_o$       & water velocity             & $\approx 10\ \mathrm{cm}\,\mathrm{s}^{-1}$           & $ 0-250\ \mathrm{cm}\,\mathrm{s}^{-1}$        \\
        $c_m$       & maximum creep              & $2\cdot 10^{-9} s^{-1}$                              & $c_m<10^{-7}s^{-1}$                           \\ \hline
    \end{tabular}
    \caption{Values are collected from the following sources \cite{Hibler1979,Leppaeranta2011, Brown1980, Mcphee1982, Martinson1990}.}
    \label{tab:values}
\end{table*}
\section{Least-Squares Formulation}\label{sec:LS}

The principle of the least-squares finite element method is to minimise the residual of the partial differential equation in first-order form. To this end, we will approximate the velocity $u: \Omega \rightarrow \mathbb R^2$ and the stresses $\sigma: \Omega \rightarrow \mathbb R^{2\times 2}$ simultaneously.
Considering a single time step with a backwards Euler scheme, where $u_{\mathrm{old}}$ denotes the solution at the previous time step and $\Delta t$ the step size, the sea-ice problem then reads
\begin{subequations}
    \label{eq:seaiceeuler}
    \begin{align}
        h\rho \frac{u - u_\mathrm{old}}{\Delta t} - \div \sigma +  \tau_o(u) & = 0.    \\
        \sigma - 2\eta \varepsilon(u)                                   & = 0     \ .
    \end{align}
\end{subequations}
For $h \neq 0$,  this leads to the time discrete problem
\begin{subequations}
\label{eq:timediscrete}
\begin{align}
    \beta^{-1/2}(u-u_{\mathrm{old}})
    -  \beta^{1/2}  \div\sigma + \beta^{1/2} \tau_o(u) & = 0 \\
    (2\eta)^{-1/2}\sigma - (2\eta)^{1/2} \varepsilon(u)                          & = 0
\end{align}
\end{subequations}
with $\beta = \Delta t (h\rho)^{-1}$, where 
\begin{align}
\label{eq:barbeta}
   \bmin \leq  {\beta} \leq\bmax
\end{align}
holds.  To simplify the notation 
we set $
\bmax=\bmin^{-1} \ .$
The advantage of this scaling will become clear in Section \ref{sec:time}, as we will
be interested in the minimisation of 
\begin{align}
    \Vert \beta^{-1/2}(u-u_{\mathrm{old}})  + \beta^{1/2} R_1 \Vert^2_{L^2(\omega)}
    + \Vert R_2  \Vert^2_{L^2(\omega)}\ .
    \label{eq:functionalH}
\end{align}
with $$R_1(u,\sigma) =   \tau_o(u) -
     \div\sigma $$ and  
$$
R_2(u,\sigma) = ({2\eta})^{-1/2}
    \sigma -  ({2\eta})^{1/2}\varepsilon(u)$$
in $H^1_{\Gamma_D}(\omega)\times H_{\Gamma_N}(\div;\mathbb S)$ with 
the usual Sobolev spaces
    \begin{align*}
        H^{1}(\Omega; \mathbb R^d) & = \{ v \in L^2(\Omega; \mathbb R^d) :
        \Vert\nabla v\Vert_{L^2(\Omega; \mathbb R^{d \times d})}< \infty \} ,  \\
        H^{1}_\Gamma(\Omega; \mathbb R ^d)  & = \{v \in H^{1}(\Omega; \mathbb R^d) : v_{|\Gamma} = 0  \}, \\
        H(\div; \Omega)                       & = \{\tau \in L^2(\Omega; \mathbb R^{d}) : \ 
        \text{and}\ \Vert \div \tau\Vert_{L^2(\Omega)} < \infty\}, \\
        H_{\Gamma}(\div; \Omega) & = \{\tau \in H(\div; \Omega) : \tau_{|\Gamma} \cdot \mathbf{n} = 0  \}, \\
        H_\Gamma(\div; \mathbb S) &= \{ \tau \in H_\Gamma(\div; \Omega)^2 : \tau = \tau^T  \} \ .
    \end{align*} 
To shorten the notation, we will also denote by $\Vert\cdot\Vert$ the norm $\Vert\cdot\Vert_{L^2(\omega)}$, where $\omega$ is the active domain introduced in the equation \eqref{eq:active_set}. Then,
the natural norm for a tuple $(u,\sigma) \in H^1_{\Gamma_D}(\omega)\times H_{\Gamma_N}(\div;\mathbb S)$ is defined by 
\begin{align}
    ||| (u,\sigma) |||^2 &= 
     \Vert u \Vert_{H^1(\omega)}^2 + \Vert \sigma\Vert^2_{H(\div;\omega)} \notag \\
   &= \Vert u \Vert^2 +  \Vert  \nabla u \Vert^2 +\Vert \sigma\Vert^2  +\Vert\div \sigma\Vert^2\ .
\end{align}
This choice of spaces is motivated by the following theorem.
\begin{theorem}
    For $u, u_{\mathrm{old}} \in H^1(\omega;\mathbb R^2)$ and $\sigma \in H(\div;\mathbb S)$,
    and $v_o \in H^1(\Omega;\mathbb R^2)$,
    \begin{align*}
        R(u,\sigma; u_{\mathrm{old}}) \in L^2(\omega)^2 \times H(\div;\omega)^2
    \end{align*}
    holds for left side 
    \begin{align}
\label{eq:R}
    R(u,\sigma; u_{\mathrm{old}}) = \begin{pmatrix}
         \beta^{-1/2} (u - u_{\mathrm{old}}) 
   + \beta^{1/2}R_1(u,\sigma) \\
       R_2(u,\sigma)
    \end{pmatrix} \ .    
\end{align}
    of the first order system \eqref{eq:timediscrete}.
\end{theorem}
\begin{proof}
    Surely $R_1(u,\sigma)  \in H(\div;\mathbb S)$. 
    Moreover, it holds $$\Vert \tau_o(u)\Vert = \Vert |u-v_o|(u-v_o)\Vert = \Vert |u-v_o|^2\Vert,$$ such that
    $\tau_o(u) \in L^2(\omega)^2$, since $(u-v_o) \in H^1(\omega) \hookrightarrow L^4(\omega)$. Therefore,
    $u - u_{\mathrm{old}} + R_1(u,\sigma) \in L^2(\omega)^2 \ .$
\end{proof}
As the ocean current reads $\tau_o(u) = \rho_oC_o|u-v_o|(u-v_o)$ it is convenient to shift the variable $u$ by $v_o$, that is to define $\tilde u = u - v_o$. The system \eqref{eq:R} then reads 
\begin{align}
\label{eq:seaiceeulershift}
    \widetilde R(u,\sigma;v_o, u_{\mathrm{old}}) = \begin{pmatrix}
         \beta^{-1/2} (u - u_{\mathrm{old}}) 
    - \tilde R_1(u,\sigma) \\
        R_2(u,\sigma) 
        +(2\eta)^{1/2} \varepsilon(v_0)
    \end{pmatrix} =0      
\end{align}
with $\tilde\tau_o(u) = \rho_oC_o|u|u$ and $\tilde R_1(u,\sigma) = \tilde \tau_o(u) -  \text{div }\sigma$.
\section{Well-posedness of the stationary problem}\label{sec:station}

This section focus on the well-posedness of the corresponding stationary problem. In fact, 
neglecting $u-u_\mathrm{old}$ in the first-order system \eqref{eq:seaiceeulershift}, leads to the minimisation of 
\begin{align}
    \mathcal{F}(u,\sigma; v_o) :=
    \Vert 
    \tilde R_1(u,\sigma) \Vert^2_{L^2(\omega)}
    + \Vert  R_2(u,\sigma)- \varepsilon(v_o)  \Vert^2_{L^2(\omega)}\ 
    \label{eq:functionalF}
\end{align}
in $H^1_{\Gamma_D}(\omega)\times H_{\Gamma_N}(\div;\mathbb S)$. The purpose of this section is to prove that $\mathcal{F}(u,\sigma; 0)$ is indeed an elliptic functional for small velocities.  
To this end, the following lemma states important properties of the ocean current that holds under the assumption that the current velocity approximation is close enough to the ocean current. 
To this end, we introduce the environment
\begin{align*}
    \mathcal U_\epsilon(v_0) = 
    \{
    u \in H^1_{\Gamma_D}(\omega;\mathbb R^2) \ : \ 
    \Vert
    u-v_o
    \Vert
    <
    \epsilon
        \} \ .
\end{align*}
Note that $u\in \mathcal U_\epsilon(v_0)$ is a reasonable assumption since the ocean current is by 
order of magnitude the dominating force (see \cite{Leppaeranta2011}) 
in the momentum equation \eqref{eq:momentum}. 

{\color{black}
\begin{lemma}
    \label{lem:L}
    Let $v_0 \in H^1(\omega)$. For any $L_1,L_2>0$ there exist $\epsilon>0$, such that
    \begin{align}
        || \tau_o(u) ||^2 \leq L_1 ||u||^2
        \label{eq:L} \ 
    \end{align}
    and
    \begin{align}
        || \tau_o(u)-\tau_o(v) ||^2 \leq L_2 ||u-v||^2
        \label{eq:L2}
    \end{align}
    holds for any $u \in \mathcal U_\epsilon(v_0)$, i.e.  
    $\tau$ is Lipschitz-continuous, with constant $L_2>0$ in  
    $\mathcal U_\epsilon(v_0)$    
    holds.
\end{lemma}

\begin{proof}
    In our case, we simply have $L_1 = \rho_o^2 C_o^2 \epsilon^2$, but the rest of the paper holds true under the more general assumption $\eqref{eq:L}$.
    The Lipschitz continuity can be proven by the following chain of estimates
    \begin{align*}
        \Vert \tilde\tau_o(u) - \tilde\tau_o(v)\Vert 
        &\leq \rho_o C_o\int_0^1 \Vert \frac{\dd}{\dd t} \tilde\tau_o (t u + (1-t)v) \Vert\dd t \\
        &\leq 2\rho_o C_o \int_0^1 \Vert tu+(1-t)v\Vert\dd t \ \Vert u-v\Vert \\
        &\leq 2\rho_o C_o (\Vert u\Vert+\Vert v\Vert)\Vert u-v\Vert \\
        &\leq 4 \rho_o C_o\epsilon \Vert u - v \Vert \  .
    \end{align*}
    The Lipschitz constant is therefore $L_2:= 16 \rho_o^2 C_o^2 \epsilon^2$.
\end{proof}
For $u \in H^1_{\Gamma_D}(\omega ; \mathbb R^2)$, as $\omega$ is assumed to be a bounded Lipschitz domain and $\emptyset \neq \Gamma \subset \partial \omega$ holds, we will also repeatedly combine the previous result with the Korn's inequality
    \begin{align}
        \Vert \varepsilon(u)\Vert \korn \geq \Vert u \Vert_{H^1(\Omega)} \quad \text{for all} \ v \in H^1_\Gamma(\Omega; \mathbb R^d)
        \label{eq:Korn}
    \end{align}
to obtain 
    \begin{align}
        || \tau_o(u) ||^2 \leq L_1\korn \Vert \varepsilon(u)\Vert^2 \ .
        \label{eq:LK} \ 
    \end{align}

\begin{corollary}\label{cor:beta}
    Let $\beta:\mathbb R^2 \rightarrow {\mathbb R}$ with $\beta(x)\leq \bmax<\infty$ for all $x\in \mathbb R^2$.
    Then, given $\kappa \in (0,1)$, there exists $\epsilon>0$ small enough such that 
    \begin{align*}
        \Vert u\Vert^2 + 2 (\beta u, \tau_0(u)) \geq \kappa \Vert u \Vert^2 
    \end{align*}
    holds for any  $u \in \mathcal U_\epsilon(v_0)$.
\end{corollary}
\begin{proof}
Let $L_1$ be chosen such that
\begin{align}
    L_1 < \frac{(\kappa -  1)^2}{4\bmax} \label{eq:assL1}
\end{align}
holds. Then, $  (\kappa -  1)^2-4L \bmax>0$ and 
$$\tilde C = {\frac {1}{2 \bmax} 
\left( 1-\kappa+
\sqrt { (\kappa -  1)^2} \right)-4L \bmax  } > 0 \ .$$
With the previous lemma, this implies that there exists $\epsilon$ such that
    \begin{align*}
        \Vert u\Vert^2 + 2 (\beta u, \tau_0(u)) 
         & \geq \Vert u \Vert^2 - \bmax  \tilde C\Vert u\Vert^2 
         -
         \frac{1}{\tilde C}\Vert\tau_0(u)\Vert^2  \\
          & \geq \Vert u \Vert^2 \left( 1- \bmax\tilde C  - \frac{L_1}{\tilde C}\right)
          \\          & \geq \kappa \Vert u \Vert^2 
    \end{align*}
    for any $u \in \mathcal U_\epsilon(v_0)$
\end{proof}

\begin{corollary}
\label{cor:R1}
For any $L$, there exists $\epsilon>0$ such that 
\begin{align*}
  \Vert \tilde R_1(u,\sigma) \Vert^2  & \geq -
      \Vert \varepsilon(u)
    \Vert^2 L\korn 
    + \frac 1 2
     \Vert 
     \div \sigma
    \Vert^2  
\end{align*}
and 
\begin{align*}
\Vert u \Vert^2 + 2(u,\tau_o(u)) +  \Vert \tilde R_1(u,\sigma) \Vert^2  & \geq (1-L)
      \Vert u
    \Vert^2 
    + \frac 1 2
     \Vert 
     \div \sigma
    \Vert^2  
\end{align*}
holds
for any $ (u,\sigma) \in
\mathcal{U}_\epsilon(0)\times H_{\Gamma_N}(\div;\Omega)$ 
\end{corollary}
}
    \begin{proof}
Since 
\begin{align*} 
    \Vert \tilde R_1(u,\sigma) \Vert^2
    &= 
     \Vert 
     \tilde\tau_o(u)
- \div \sigma
    \Vert^2
 \geq 
     - \Vert  
      \tilde  \tau_o(u)
    \Vert^2 
    + \frac 1 {2}
    \Vert  \div \sigma
    \Vert^2  \ 
    \end{align*}
holds, the equation \eqref{eq:LK} implies the first statement.
Its combination with corollary \ref{cor:beta} leads to the second statement.
\end{proof}

\begin{lemma}
\label{thm:ellipticitylemma}
    Assume that there exists a constant $\alpha \in [0,2]$ and $k>0$ such that 
\begin{align} 
\label{eq:ellipticitylemma}
    \Vert R_2(u,\sigma) \Vert^2
    &\geq
\alpha
 \Vert \sigma \Vert^2
+ (\alpha-k C\korn)
    \Vert
\varepsilon(u)
    \Vert^2   
    - \frac{1}{C} \Vert \div \sigma  \Vert^2
    \end{align}   
    holds for any constant $C>0$, where $\korn$ denotes the Korn constant from \eqref{eq:Korn}.
    Then, there exists $\epsilon>0$ such that 
\begin{align*}
\mathcal F(u,\sigma,0)
  & \geq 
  \frac{\underline{\eta}^2}{8k\korn}
      \Vert \varepsilon(u)
    \Vert^2  
    + \frac{\underline{\eta}^2}{2k\korn}
     \Vert 
     \sigma
    \Vert^2  
    + \frac 1 4
     \Vert 
     \div \sigma
    \Vert^2  
\end{align*}
holds
for any $ (u,\sigma; 0) \in
\mathcal{U}_\epsilon(0)\times H_{\Gamma_N}(\div;\Omega)$ .
\end{lemma}
\begin{proof}
    Combining \eqref{eq:ellipticitylemma} with the result of the previous Corollary \ref{cor:R1}, we obtain that for any $L>0$, there exists $\epsilon$ such that
\begin{align*}
  A \Vert  \tilde R_1 \Vert^2
  +\Vert \tilde R_2 \Vert^2 & \geq
  \alpha
 \Vert \sigma \Vert^2+
 \gamma_1
      \Vert \varepsilon(u)
    \Vert^2 
    + \left(\frac A 2
- \frac{1}{C}
    \right)
     \Vert 
     \div \sigma
    \Vert^2  
\end{align*}
holds for any $A>0$
with $\gamma_1=\alpha-k C\korn - A L_1\korn\ $.
In order for $\gamma_1$ to be positive, we notice that $\tilde \gamma_1 = \alpha-kC\korn$ has to be positive. We can therefore set $L$ such that
$$ \gamma_1 = \frac{\alpha-kC\korn}{2}$$
i.e.
$$ L = \frac{\alpha-kC\korn}{2A\korn} \ .$$
With can now choose $C$ such that $\tilde \gamma_1$ is positive and the results then holds for any $A>2C^{-1}$, for instance for 
\begin{align*}
    C=\frac{\alpha}{2k\korn}, \qquad
    A = \frac{4 }{\alpha k\korn} \ .
\end{align*}
Plugging the expressions of $C$ and $A$ in $\gamma_1$ finishes the proof.
\end{proof}

\begin{theorem}
There exists $\epsilon>0$ such that the functional
$  
    \mathcal{F}(u,\sigma; 0) 
$ 
    is elliptic in $\mathcal{U}_\epsilon(0)\times H_{\Gamma_N}(\div;\mathbb S)$ 
    , i.e. there exists a constant $C_{E,\mathcal F}>0$ such that 
$$\mathcal{F}(u,\sigma; 0) \geq C_{E,\mathcal F} ||| (u,\sigma) |||^2$$
holds for any $(u,\sigma) \in \mathcal{U}_\epsilon(0)\times H_{\Gamma_N}(\div;\mathbb S)\ .$ 
\end{theorem}

    \begin{proof}
Integration by parts leads to 
\begin{align*} 
    \Vert R_2(u,\sigma) \Vert^2
    &= \Vert
    (2 \eta)^{-1/2}
\sigma
    \Vert^2
    +
    \Vert
    (2 \eta)^{1/2}
\varepsilon(u)
    \Vert^2
    - 2 (\varepsilon(u),\sigma ) \\
    &\geq 2\underline{\eta} 
\Vert \sigma \Vert^2
+ 2\underline{\eta} 
    \Vert
\varepsilon(u)
    \Vert^2
    + 2(u,\div \sigma )
\\
    & \geq
    2\underline{\eta} 
 \Vert \sigma \Vert^2
+ 
2\underline{\eta}
    \Vert
\varepsilon(u)
    \Vert^2
    - 
    C \Vert u \Vert^2
    - \frac{
    1}{C} \Vert \div \sigma  \Vert^2
    \end{align*}
    for any 
    $C>0$
    .
    Using the Korn inequality \eqref{eq:Korn} one obtains 
\begin{align*} 
    \Vert R_2(u,\sigma)  \Vert^2
    &\geq
2\underline{\eta}
 \Vert \sigma \Vert^2
+ \left( 2\underline{\eta} -C\korn\right)
    \Vert
\varepsilon(u)
    \Vert^2   
    - \frac{1}{C} \Vert \div \sigma  \Vert^2 \ .
    \end{align*}    
  %
  Noticing that this is the assumption of Lemma \eqref{eq:ellipticitylemma} for $k=1$ finishes the proof.
    \end{proof}

\section{{Well-posedness} of the time dependant solution}\label{sec:time}

This section is dedicated to the norm equivalence of the homogeneous Least-Squares  functional 
$\mathcal{H}(u, \sigma; 0,0)$ where
\begin{align}
    \mathcal{H}(u, \sigma; v_o, u_{\mathrm{old}})
    &=\Vert \beta^{-1/2} (u - u_{\mathrm{old}}) 
    + \beta^{1/2}\widetilde R_1(u,\sigma) \Vert^2 \notag\\
&+ \Vert R_2(u,\sigma) -  (2\eta)^{1/2}\varepsilon(v_o)\Vert^2
\label{eq:LSH}
\end{align}
to the norm of the space (Theorem \ref{thm:normequi}). Recall that if the Least-Sqaures functional would be linear, this would lead to the reliability and efficiency of $\mathcal{H}$. However, $\mathcal H$ is not linear and the result is stated in Theorem \ref{thm:2}.
The key point is that the constants in the Least-Squares functional are chosen such that the mixed term involving $u$ and $\div \sigma$ cancels out during the integration by parts, as shown in the following Lemma.
\begin{lemma}
For any $u\in \mathcal{U}$ and $\sigma \in H_{\Gamma_N}(\div; 
\Omega)$, it holds
\label{lem:nomixedterm}
    \begin{align*}
        \mathcal{H}(u, \sigma;&0,0)
      =  \Vert \beta^{-1/2} u
        \Vert^2 + 
        \Vert \beta^{1/2} \widetilde R_1(u,\sigma)\Vert^2 +\Vert(2\eta)^{-1/2} \sigma\Vert^2 \notag \\
        &
        + \Vert(2\eta)^{1/2}\varepsilon(u)\Vert^2 - 2(u ,\tilde \tau_o(u)) - 2(\mathrm{as}(\sigma),\nabla(u))     
        \ . 
    \end{align*}
\end{lemma}
\begin{proof}
It holds
  \begin{align*}
        \mathcal{H}(u, \sigma;0,0)
        = & \Vert \beta^{-1/2} u
        \Vert^2 + 
        \Vert \beta^{1/2} \widetilde R_1(u,\sigma)\Vert^2 +\Vert(2\eta)^{-1/2} \sigma\Vert^2 \notag \\
        &
        + \Vert(2\eta)^{1/2}\varepsilon(u)\Vert^2 - 2(u ,\widetilde R_1) - 2(\sigma,\varepsilon(u))  \ .
    \end{align*}
Inserting the definition of $\widetilde R_1$ and integration by parts finishes the proof.
\end{proof}

\begin{theorem}\label{thm:normequi}
    There exists $\epsilon>0$ small enough such that it holds
    \begin{align}
         C_1 ||| (u, \sigma)|||^2  \leq
       \mathcal H(u,\sigma;0, 0)  \\
        \text{and }
        \mathcal H(u,\sigma;0, 0) \leq C_2||| (u, \sigma)|||^2
        \label{eq:6.1}
    \end{align}  
    for any $u\in \mathcal{U}_\epsilon(0)$ and $\sigma \in H_{\Gamma_N}(\div; \mathbb S)$, 
    with 
    \begin{align*}
        &C_1 := \min \left\{ \frac 1 2 \bmin,{2\etamin}{\korn^{-1}},{2\etamin}\right\} ,\\ 
        &C_2 := \max\left\{4\bmax, 4\etamax \right\} \ .
    \end{align*}
\end{theorem}

\begin{proof}
Combining Lemma \ref{lem:nomixedterm} 
for a symmetric stresses $\sigma$ and Corollary \ref{cor:R1}, there exists $\epsilon$ such that
    \begin{align}
    H(u,\sigma;0, 0) 
        \geq & \Vert \beta^{-1/2} u\Vert^2  - L \Vert  u\Vert^2
        + \frac 1 2 \Vert\beta^{1/2}\div\sigma\Vert^2 \notag \\ &+\Vert(2\eta)^{-1/2} \sigma\Vert^2 +\Vert(2\eta)^{1/2}\varepsilon(u)\Vert^2 
        \end{align}
        holds for any given $L>0$ to be chosen later. This leads to
        \begin{align*}
        H(u,\sigma;0, 0) 
        \geq & \Vert u\Vert^2 (\bmin - L) 
        + \frac 1 2 \bmin  \Vert\div\sigma\Vert^2 
        +2 \etamin \Vert\sigma\Vert^2+2\etamin \Vert\varepsilon(u)\Vert^2 \ .
    \end{align*}
    We can therefore choose $L=\frac 1 2 \bmin$ and 
    Korn's inequality concludes the proof of the lower bound.
    For the second inequality of \eqref{eq:6.1}, the triangle inequality and Lemma \ref{lem:L} 
    imply the following chain of estimate
    \begin{align*}
        \mathcal{H}(u&,\sigma;0,0) \\
        \leq 
        & 2\left(\bmax(\Vert u\Vert^2 +\Vert\div\sigma\Vert^2+\Vert \tau_o(u)\Vert^2) +2\etamax^2(\Vert\sigma\Vert^2+\Vert\varepsilon(u)\Vert^2) \right)   \\
        \leq &2\left(2\bmax( \Vert u\Vert^2
        +\Vert\div\sigma\Vert^2)+2\etamax(\Vert\sigma\Vert^2+\Vert\varepsilon(u)\Vert^2 )\right) \\
        \leq & \max\left\{4\bmax, 4\etamax \right\} 
        \left(\Vert u\Vert_{H^1(\omega)}^2 + \Vert\sigma\Vert^2_{H(\mathrm{div};\omega)}\right)\ .
    \end{align*}
\end{proof}
\begin{theorem}[Reliability and Efficiency]\label{thm:2}
    Let \\ $(u,\sigma) \in \mathcal U_{\tilde\epsilon}(0) \times H_{\Gamma_N}(\div,\mathbb S)$ 
    solve \eqref{eq:seaiceeulershift} for a given $v_o \in H^1_0(\Omega; \mathbb R^2)$.
    Under the assumptions of Theorem \ref{thm:normequi}, there exits $\tilde\epsilon>0$ and constants
$C_{\text{rel},\mathcal H},C_{\text{eff},\mathcal H}>0$
    such that
    \begin{align*}
        C_{\text{rel},\mathcal H} &||| (u-v), (\sigma - \tau)|||^2
         \leq \mathcal{H}(v, \tau; v_o, u_{\mathrm{old}}) \\
         &\text{ and }
  \mathcal{H}(v, \tau; v_o, u_{\mathrm{old}})       
         \leq C_{\text{eff},\mathcal H}   ||| (u-v), (\sigma - \tau)|||^2 
    \end{align*}
    holds for any 
    $v \in  \mathcal{U}_{\tilde\epsilon}(0)$ and $\tau \in H_{\Gamma_N}(\div,\mathbb S)$ .
\end{theorem}
\begin{proof}
    First we show that there exists $\kappa \in [0,C_1)$ such that
    \begin{align}
        \Vert \beta^{1/2} ( \tau_0(u-v) +\tau_0(u)-\tau_0(v))\Vert^2 \leq \kappa \Vert u - v \Vert_{H^1(\omega)} \ .\label{ineq:kappa}
    \end{align}
    To this end, notice that $u-v \in \mathcal{U}_{2\tilde\epsilon}(0)$.
    By Lemma \ref{lem:L} to $\mathcal{U}_{2\tilde\epsilon}(0)$, we can therefore conclude the following estimates
    \begin{align*}
        \Vert \beta^{1/2} (\tau_0(u-v) +&\tau_0(u)-\tau_0(v))\Vert^2 \\
        &\leq \bmax\Vert \tau_0(u-v) +\tau_0(u)-\tau_0(v)\Vert^2                           \\
        & \leq 2 \bmax \left(\Vert \tau_0(u-v)\Vert^2 + \Vert \tau_o(u)-\tau(v)\Vert^2\right) \\
        & \leq 2\bmax \left( L_1 \Vert u -v \Vert^2 + L_2 \Vert u -v\Vert^2\right)  \\
        & \leq 2\bmax (L_1 + L_2) \Vert u -v\Vert^2 \ .
    \end{align*}
    Choosing $\tilde\epsilon$ small enough such that $\kappa = 2\bmax (L_1 + L_2) < C_1$ allow to 
    obtain \textbf{reliability} from Theorem \ref{thm:normequi}: 
    \begin{align*}
         & C_1 |||(u-v), (\sigma - \tau)|||^2 \leq \mathcal{H}(u-v, \sigma - \tau; 0) \\
         & \leq \Vert \beta^{-1/2}(u-u_{\mathrm{old}}) - \beta^{-1/2} (v-u_{\mathrm{old}}) 
         -\beta^{1/2}\div(\sigma-\tau)  \\ 
         &- \beta^{1/2}(\tau_o(u-v) +\tau_o(u) - \tau_o(u) 
         + \tau_o(v) -  \tau_o(v) )\Vert^2 \\
         & + \Vert (2\eta)^{1/2}\varepsilon(u-u_h) -(2\eta)^{-1/2} (\sigma-\tau)\Vert^2 \\
         & \leq 
         \mathcal H(u, \sigma;  v_o, u_{\mathrm{old}} )
         \\
     &+ \mathcal{H}(v,\tau; v_o, u_{\mathrm{old}}) + \Vert \beta(\tau_0(u-v) -\tau_0(u)+\tau_0(v))\Vert^2       \ .
\end{align*}
    Since $\mathcal H(u, \sigma;  v_o, u_{\mathrm{old}})=0$ equation \eqref{ineq:kappa} now implies
     \begin{align*}
        C_1 |||(u-v), (\sigma - \tau)|||^2 
        & {\leq} \mathcal{H}(v,\tau; v_o, u_{\mathrm{old}}) + \kappa \Vert u-v\Vert^2_{H^1(\omega)} .
    \end{align*}
    
    Subtracting $\kappa \Vert u - v \Vert^2_{H^1(\omega)}$ from the inequality yields to the reliability estimate
    \begin{align*}
        (C_1-\kappa) \Vert u - v\Vert^2_{H^1(\omega)} + C_1 \Vert \sigma-\tau\Vert^2_{H(\div;\omega)} \leq \mathcal{H}(v, \tau; v_o, u_{\mathrm{old}}).
    \end{align*}
    which is the statement of the theorem for $C_{\text{rel},\mathcal H} := C_1 - \kappa > 0$ .
    On the other side, Equation {\eqref{ineq:kappa}} implies the \textbf{efficiency}
    \begin{align*}
         & \mathcal H(v, \tau;  v_o, u_{\mathrm{old}}) \\
         & = \beta^{-1/2}\Vert (u-v) - \beta^{1/2} \div (\sigma - \tau) \\
         &+ \beta^{1/2} (\tau_o(v) - \tau_o(u) + \tau_o(u-v) - \tau_o(u-v))\Vert\\
          &+\Vert (2\eta)^{1/2}\varepsilon(u-v) -(2\eta)^{-1/2} (\sigma-\tau)\Vert^2 \\
         & \leq \mathcal{H}(u-v, \sigma -\tau;0) + \Vert \beta\tau_o(v) - \tau_o(u) +\tau_o(u-v))\Vert^2 \\
         & {\leq} (\kappa + C_2)\left( \Vert u-v\Vert^2_{H^1(\omega)} + \Vert\sigma-\tau\Vert^2_{H(\div,\omega)} \right).
    \end{align*}
with the efficiency constant $C_{\text{eff},\mathcal H}:=(\kappa + C_2) $.
\end{proof}

\section{Estimates for non-symmetric stresses}\label{sec:nonsym}
The ellipticity and continuity of the Least-Squares functional proven previously is restricted to the use of non-symmetric stresses. As it can be numerically expensive to use symmetric finite element spaces, the purpose of this section is to extend the well-posedness to the case of non-symmetric stresses.  To this end, we infer that 
for any $\in H^1_{\Gamma_D}(\omega)\times H_{\Gamma_N}(\div;\omega)$
$$
\begin{aligned}
&\left\|{(2 \eta)^{-1/2}
(\tau}-{\tau}^\top)\right\|
\\ &\qquad =\left\|((2 \eta)^{-1/2}{\tau}-(2 \eta)^{1/2}{\epsilon}({v}))
-((2 \eta)^{-1/2}{\tau}-(2 \eta)^{1/2}{\epsilon}({v}))^\top \right\| \\
&
 \qquad  \leq 2\|(2 \eta)^{-1/2}{\tau}-(2 \eta)^{1/2}{\epsilon}({v})\|  \ .
\end{aligned}
$$
This implies
\begin{align}
\label{eq:nonsymkey}
\left\|(2 \eta)^{-1/2} \mathrm{as}
(\tau) \right\|^2 
=\frac 1 4
\left\| (2 \eta)^{-1/2}
(
\tau-{\tau}^\top
)\right\|^2
\leq \|R_2 \|^2 \ .
\end{align}
and leads to the well-posedness for both the stationary case and the time-dependant case, as shown in the following two theorems. 
\begin{theorem}
There exists $\epsilon>0$ such that the functional
$  
    \mathcal{F}(u,\sigma; 0) 
$ 
    is elliptic in $\mathcal{U}_\epsilon(0)\times H_{\Gamma_N}(\div;\Omega)$ 
    , i.e. there exists a constant $C_{E,\mathcal F}^{\mathrm{as}}>0$ such that 
$$\mathcal{F}(u,\sigma; 0) \geq C_{E, \mathcal{F}}^{\mathrm{as}} ||| (u,\sigma) |||^2$$
holds for any $(u,\sigma) \in \mathcal{U}_\epsilon(0)\times H_{\Gamma_N}(\div;\Omega)\ .$ 
\end{theorem}

    \begin{proof}
The integration by part of non-symmetric stresses leads to  
\begin{align*} 
    \Vert R_2 \Vert^2
    &= \Vert
    (2 \eta)^{-1/2}
\sigma
    \Vert^2
    +
    \Vert
    (2 \eta)^{1/2}
\varepsilon(u)
    \Vert^2
    - 2 (\varepsilon(u),\sigma ) \\
    &\geq 2\underline{\eta} 
\Vert \sigma \Vert^2
+ 2\underline{\eta} 
    \Vert
\varepsilon(u)
    \Vert^2 {\color{black} 
    + 2 ( \nabla(u), \as(\sigma))
    }
    \\& \quad
    + 2(u,\div \sigma )
   \\
    & \geq
    2\underline{\eta} 
 \Vert \sigma \Vert^2
+ 2\underline{\eta}
    \Vert
\varepsilon(u)
    \Vert^2
    - 
    C \Vert u \Vert^2\\&
    - \frac{1}{C} \Vert \div \sigma  \Vert^2
    -  C\Vert \nabla(u)\Vert^2 
    - \frac 1 C \Vert  \as(\sigma)\Vert^2 
    \end{align*}
    for any $C>0$
    . Since 
    $$   - \frac 1 C \Vert  \as(\sigma)\Vert^2  \geq - \frac {\color{black} 2\etamin} C \Vert (2\eta)^{-1/2} \as(\sigma)\Vert^2 \ ,$$
    using \eqref{eq:nonsymkey} and the Korn inequality \eqref{eq:Korn} one more time leads to
\begin{align} 
\label{eq:ellipticity2NS}
    \Vert R_2 \Vert^2 \left(1 + \frac {\color{black} 2\etamin} C \right)
    &\geq
2\alpha\underline{\eta}
 \Vert \sigma \Vert^2
+ \gamma_1
    \Vert
\varepsilon(u)
    \Vert^2   
    - \frac{1}{C} \Vert \div \sigma  \Vert^2
    \end{align}    
 with $\gamma_1 = 2\underline{\eta}-2 C \korn $  .
Therefore, the fact that the assumptions of Lemma \ref{thm:ellipticitylemma} are satisfied with $\alpha=2\underline{\eta}$ and $k=2$ finishes the proof.
    \end{proof}

\begin{theorem}
    There exists $\epsilon>0$ small enough such that it holds
    \begin{align*}
         C_1 ||| (u, \sigma)|||^2  \leq
       \mathcal H(u,\sigma;0, 0) \\
        \text{and }
        \mathcal H(u,\sigma;0, 0) \leq C_2||| (u, \sigma)|||^2
    \end{align*}  
    for any $u\in \mathcal{U}_\epsilon(0)$ and $\sigma \in H_{\Gamma_N}(\div; \Omega)$, 
    with 
    \begin{align*}
    &C_1 := \left( 1+ {8\etamin}\bmax\right)^{-1}\min \left\{ \frac 1 2 \bmin,{2}\etamin\korn^{-1},{2}{\etamin}\right\} ,\\ 
        &C_2 := \max\left\{4\bmax, 4\etamax \right\} \ .
    \end{align*}
\end{theorem}

\begin{proof}
Combining Lemma \ref{lem:nomixedterm} 
 and Corollary \ref{cor:R1}, there exists $\epsilon$ such that
    \begin{align*}
    H(u,\sigma;0, 0) 
        \geq & \Vert \beta^{-1/2} u\Vert^2  - L \Vert  u\Vert^2 -  C \Vert \nabla u\Vert^2
        + \frac 1 2 \Vert\beta^{1/2}\div\sigma\Vert^2 \notag \\ &+\Vert(2\eta)^{-1/2} \sigma\Vert^2 +\Vert(2\eta)^{1/2}\varepsilon(u)\Vert^2 
        - \frac 1 C \Vert \as\ \sigma\Vert^2
        \end{align*}
        holds for any given $C,L>0$ to be chosen later. This leads to
        \begin{align*}
        H(u,\sigma;0, 0) 
        \geq & \Vert u\Vert^2 (\bmin - L-CC_k) 
        + \frac 1 2 \bmin  \Vert\div\sigma\Vert^2 \notag \\ &
        +2 \etamin \Vert\sigma\Vert^2+2\etamin \Vert\varepsilon(u)\Vert^2  -  {2\etamin} C^{-1} \Vert (2\etamin)^{-1/2}\as\ \sigma\Vert^2\ .
    \end{align*}
    We can therefore choose $L=\frac 1 4 {\korn\bmin}$, $C=\frac 1 4 \bmin$ to obtain
            \begin{align*}
        H(u,\sigma;0, 0) \left( 1+ {8\etamin}\bmax\right)
        \geq &  \frac 1 2 \bmin \left(  \Vert u \sigma\Vert^2
        +  \Vert\div\sigma\Vert^2 \right) \notag \\ &
        + 2\etamin \Vert\sigma\Vert^2+2\etamin \Vert\varepsilon(u)\Vert^2 \ .
    \end{align*}
    and 
    Korn's inequality concludes the proof of the lower bound.
\end{proof}
    
The continuity bounds, as well as the reliability and efficiency theorem \ref{thm:2} holds also using the non-symmetric stresses, as no integration by parts is required in the proofs.

\section{Finite-Element approximation}\label{sec:gauss_newton}\label{sec:FEM}
This section gives additional details on the algorithm used to solve the Sea-ice problem \eqref{eq:seaiceeuler} with the Least-Squares method. 
In fact, the nonlinear least-squares functional \eqref{eq:LSH} {\color{black}will be minimised} using the Gauss-Newton method described in Algorithm \ref{alg:gauss}.
In this approach, the nonlinear least-squares functional is replaced by a sequence of linear ones.
To this end, we first need the calculation of the Gateaux derivatives
\begin{align*}
&\mathrm{D}_u\mathcal{\tilde \tau}_1(u,\sigma)[v] 
    =  \begin{cases}  \beta^{1/2}& \text{if}\ |u| = 0 \\
     \beta^{1/2}\frac{\Delta t\rho_o C_o}{h\rho}  \Big(|u|v  
     +\frac{(u,v)}{|u|}(u)\Big) & \text{else}
    \end{cases} \\
    &\mathrm D_u\mathcal{\widetilde R}_1(u,\sigma)[v] 
    = 
     \beta^{-1/2}v + \beta^{1/2}\mathrm D_u\mathcal{\tilde \tau}_1(u,\sigma)[v] \\
    &\mathrm{D}\mathcal{\widetilde R}(u,\sigma)[v, \tau] = \begin{pmatrix} 
        \mathrm D_u\mathcal{\widetilde R}_1(u,\sigma)[v] - \beta^{1/2}\div \tau \\
        (2\eta)^{-1/2}\tau-(2\eta)^{1/2}\varepsilon(v)
    \end{pmatrix}
\end{align*}
from $\mathcal{\widetilde R}$ defined in \eqref{eq:seaiceeulershift}.
We are then solving a sequence of least-squares problems of a linearizations 
of the terms $\mathcal{\widetilde R}$. The linearized
problem read: for a fixed iterate $(u^{(k)}, \sigma^{(k)}) \in H^1_{\Gamma_D}(\omega;\mathbb R^2) \times H_{\Gamma_N}(\div;\omega)^2$ 
we minimise the functional 
\begin{align*}
    H^{(k)}(u^{(k)}, \sigma^{(k)}; v, \tau) = \left\Vert \mathcal R(u^{(k)}, \sigma^{(k)}) + D\mathcal{R}(u^{(k)}, \sigma^{(k)})[v, \tau] \right\Vert^2
\end{align*}
This problem is equivalent to the variational problem:
find $(\Delta^{(k)}, \delta^{(k)}) \in H^1_{\Gamma_D}(\omega;\mathbb R^2) \times H_{\Gamma_N}(\div; \omega)^2$ such that
\begin{align}
    (\mathrm{D}\mathcal{R}(u^{(k)},&\sigma^{(k)})[\Delta^{(k)},\delta^{(k)}], \mathrm{D}\mathcal{R}(u^{(k)},\sigma^{(k)})[v,\tau]) \notag \\
    &=-(\mathcal{R}(u^{(k)},\sigma^{(k)}),\mathrm{D}\mathcal{R}(u^{(k)},\sigma^{(k)})[v,\tau]) \label{eq:gnstep}
\end{align}
for all $(v,\tau)\in H^1_{\Gamma_D}(\omega;\mathbb R^2) \times H_{\Gamma_N}(\div; \omega)^2$. 
A possible stopping criterion for the iteration is 
\begin{align*}
    \tau_{\mathrm{stop}} = 1 - \frac{\mathcal{H}^{(k)}(u^{(k)}, \sigma^{(k)})}{\mathcal{H}^{(k-1)}(u^{(k-1)}, \sigma^{(k-1)})} \leq \mathrm{tol}
\end{align*}
for some chosen tolerance $\mathrm{tol}$. An alternative would be to use the residual of \eqref{eq:gnstep} as a stopping 
criterion.
\begin{algorithm}
    \caption{Gauss-Newton}\label{alg:gauss}
    \begin{algorithmic}
        \Require $u^{(0)}, u_{\mathrm{old}}, \mathrm{tol}$
        \State $\tau_{\mathrm{stop}} \gets \infty$
        \State $k \gets 0$
        \While{$\tau_{\mathrm{stop}} < \mathrm{tol}$}

        \State $(\Delta^{(k)}, \delta^{(k)}) \gets$ solve equation \eqref{eq:gnstep}
        \State $(u^{(k+1)}, \sigma^{(k+1)}) \gets (u^{(k)}+\Delta^{(k)}, \sigma^{(k)}+\delta^{(k)})$
        \State $\tau_{\mathrm{stop}} \gets$ evaluate stopping criteria
        \State $k \gets k + 1$

        \EndWhile
    \end{algorithmic}
\end{algorithm}

We are now in place to introduce the Finite-Element used in our computations.
The ellipticity and continuity of the Least-Squares functional proven previously also ensures the well-posedness of our Least-Squares algorithm
 in any conforming subspace. We use a conforming triangulation  $\mathcal{T}$ of the bounded Lipschitz domain $\omega$ and are free to choose compatible elements regarding their approximation properties. To this end, the next definition recalls the standard finite element spaces. 
\begin{definition}[Lagrange Finite Element Spaces]
    Let
     $\mathcal{T}$ be a triangulation of a polygonal bounded domain $\Omega\subset \mathbb R^d$ for $\ell\in\mathbb N$. The discontinuous Lagrange finite element space is given by
        \begin{align*}
        P^\ell(\mathcal T;\mathbb R^d)  = (P^\ell(\mathcal T))^d.
    \end{align*}
    with
    \begin{align*}
        &P^{\ell}(T)  =\left\{v \in L^{\infty}(T) \mid v \text { is polynomial on } T \text { of degree } \ell\right\},                           \\
        &P^{\ell}(\mathcal{T})           =\left\{v_{\mathcal T} \in L^{\infty}(\Omega)\left|\forall T \in \mathcal{T}, v_{\mathcal T}\right|_{T} 
        \in P^{\ell}(T)\right\}.
    \end{align*}
    The conforming Lagrange finite element spaces are given by
    \begin{align*}
        S^\ell(\mathcal{T})          & =  P^\ell(\mathcal{T}) \cap H^1(\Omega) ,          \\
        S_{\Gamma}^\ell(\mathcal{T}) & =  P^\ell(\mathcal{T}) \cap H^1_{\Gamma}(\Omega) .
    \end{align*}
    \end{definition}
    For the approximation of the stress we use the $H(\div)$ conforming Raviart-Thomas elements defined in the next definition \cite{Raviart1975}.
\begin{definition}
[Raviart-Thomas Finite Element Spaces]
 The Raviart-Thomas finite element space is given by 
 \begin{align*}
     RT^\ell_\Gamma(\mathcal{T}) &= RT^\ell(\mathcal{T}) \cap H_\Gamma(\div;\Omega)
 \end{align*}
 with
    \begin{align*}
        RT^\ell(\mathcal{T}) &= \{ \tau  \in H(\div;\Omega) : v_{|T}  = p(x) + x s(x),                               \\
                                    &\qquad p\in P^\ell(T;\mathbb R^d),\ s\in P^\ell(T), \text{for all} \ T \in \mathcal{T}\} \ .
    \end{align*}
\end{definition}
For discretization of our continuous spaces, we choose the following conform discrete spaces 
\begin{align*}
     S_{\Gamma_D}^{\ell+1}(\mathcal{T};\mathbb R^2) \subset H^1(\omega; \mathbb R^2), \qquad \left( RT^\ell_\Gamma(\mathcal{T})\right)^2) \subset H(\div; \omega)^2.
\end{align*}
The discrete problem of \eqref{eq:gnstep} for reads then: \\ find $(\Delta^{(k)}_h, \delta^{(k)}_h) \in  S_{\Gamma_D}^{\ell+1}(\mathcal{T};\mathbb R^2) \times \left( RT^\ell_\Gamma(\mathcal{T})\right)^2$ such that
\begin{align*}
    (\mathcal{R}(u^{(k)}_h,\sigma^{(k)}_h),\mathrm{D}\mathcal{R}(u^{(k)}_h,\sigma^{(k)}_h)[v_h,\tau_h]) = 0
\end{align*}
for all $(v_h,\tau_h)\in  S_{\Gamma_D}^{\ell+1}(\mathcal{T};\mathbb R^2) \times \left( RT^\ell_\Gamma(\mathcal{T})\right)^2$.

\section{Numerical experiments}\label{sec:experiments}

In order to demonstrate the previously proven results, this section introduce numerical experiments.
We first focus on lowest order discretizations and look for 
 find $(\Delta^{(k)}_h, \delta^{(k)}_h) \in  S_{\Gamma_D}^{2}(\mathcal{T};\mathbb R^2) \times \left( RT^1_\Gamma(\mathcal{T})\right)^2$ such that
\begin{align*}
    (\mathcal{R}(u^{(k)}_h,\sigma^{(k)}_h),\mathrm{D}\mathcal{R}(u^{(k)}_h,\sigma^{(k)}_h)[v_h,\tau_h]) = 0
\end{align*}
for all $(v_h,\tau_h)\in  S_{\Gamma_D}^{2}(\mathcal{T};\mathbb R^2) \times \left( RT^1_\Gamma(\mathcal{T})\right)^2$.
Recall that regarding the approximation properties of these spaces,  the optimal convergence rate for the 
$H^1$ error is 1 against the maximum meshsize $h_{\max}$. Therefore the optimal rate 
for the least-squares functional is 2. In our simulations, a triangular mesh is used, but note that other types of meshes are possible. The implementation was performed in NGSolve/Netgen \cite{schoberl1997,schoberl2014}.

\subsection{Square domain}
We first present an example with a manufactured stationary solution on the square domain $\Omega = \omega = [0,500\;\mathrm{km}]^2$ with a
fix sea ice height and density over time
\begin{align*}
    h = 1\; , \qquad A = x/500\; \mathrm{km}.
\end{align*}
We use uniform time stepping with $\Delta t = 600 \; \mathrm{s}$ a circular ocean current (see Figure \ref{fig:velo} (a))
\begin{align*}
     & u_o = 0.1\begin{pmatrix}
        (2y-500)/500 \\ (500-2x)/500
    \end{pmatrix}
\end{align*}
After about 4 hours the discrete time derivative becomes negligible and the problem
reaches an equilibrium. The solution after 1 day is plotted in Figure \ref{fig:velo} (a).
\begin{figure}[ht]
    \centering

    \begin{minipage}[t]{0.48\textwidth}
        \centering
        \includegraphics[scale=0.1]{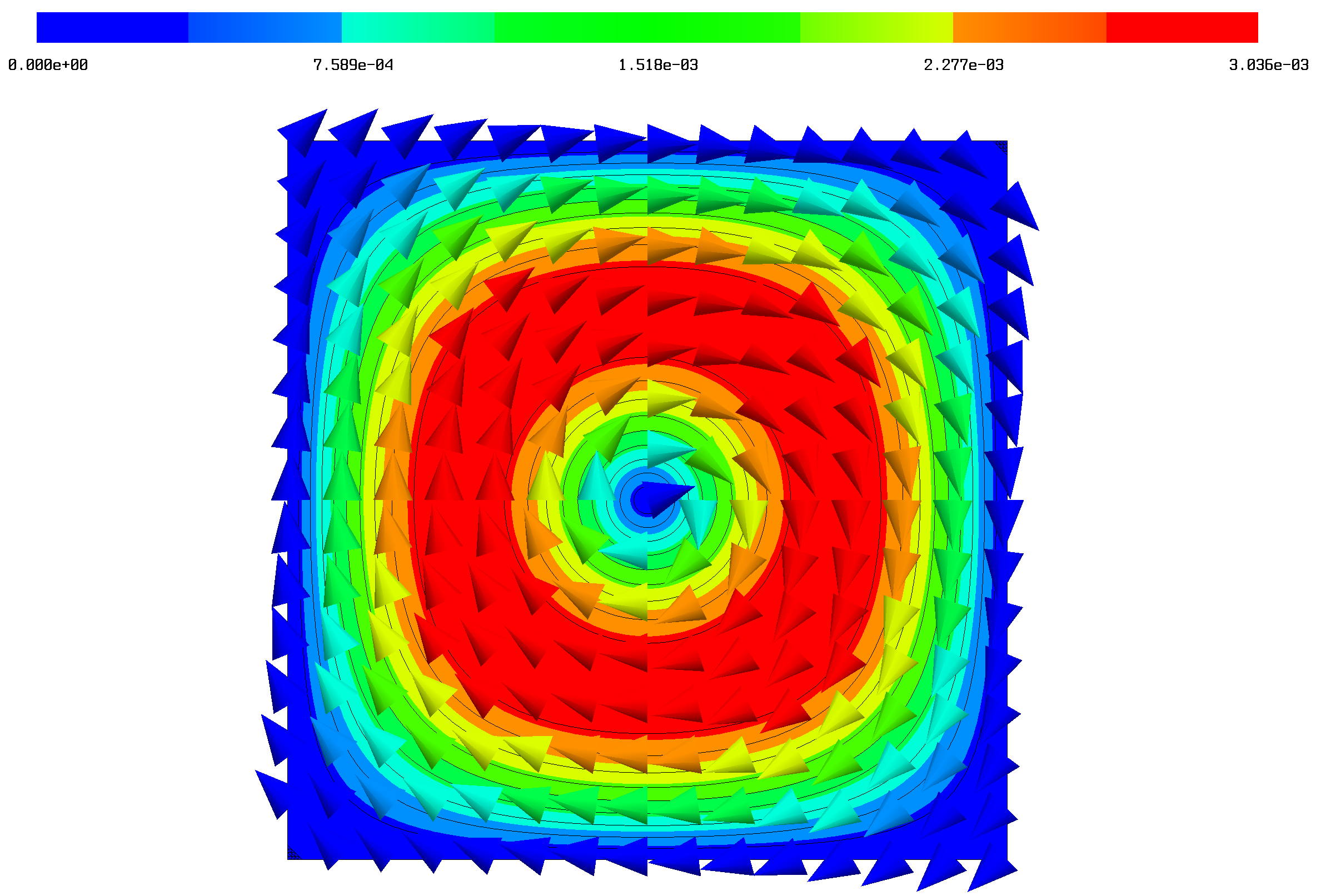}\\
        Velocity $u$ after 1 day (a)
    \end{minipage}
    \begin{minipage}[t]{0.48\textwidth}
        \centering
        \includegraphics[scale=0.1]{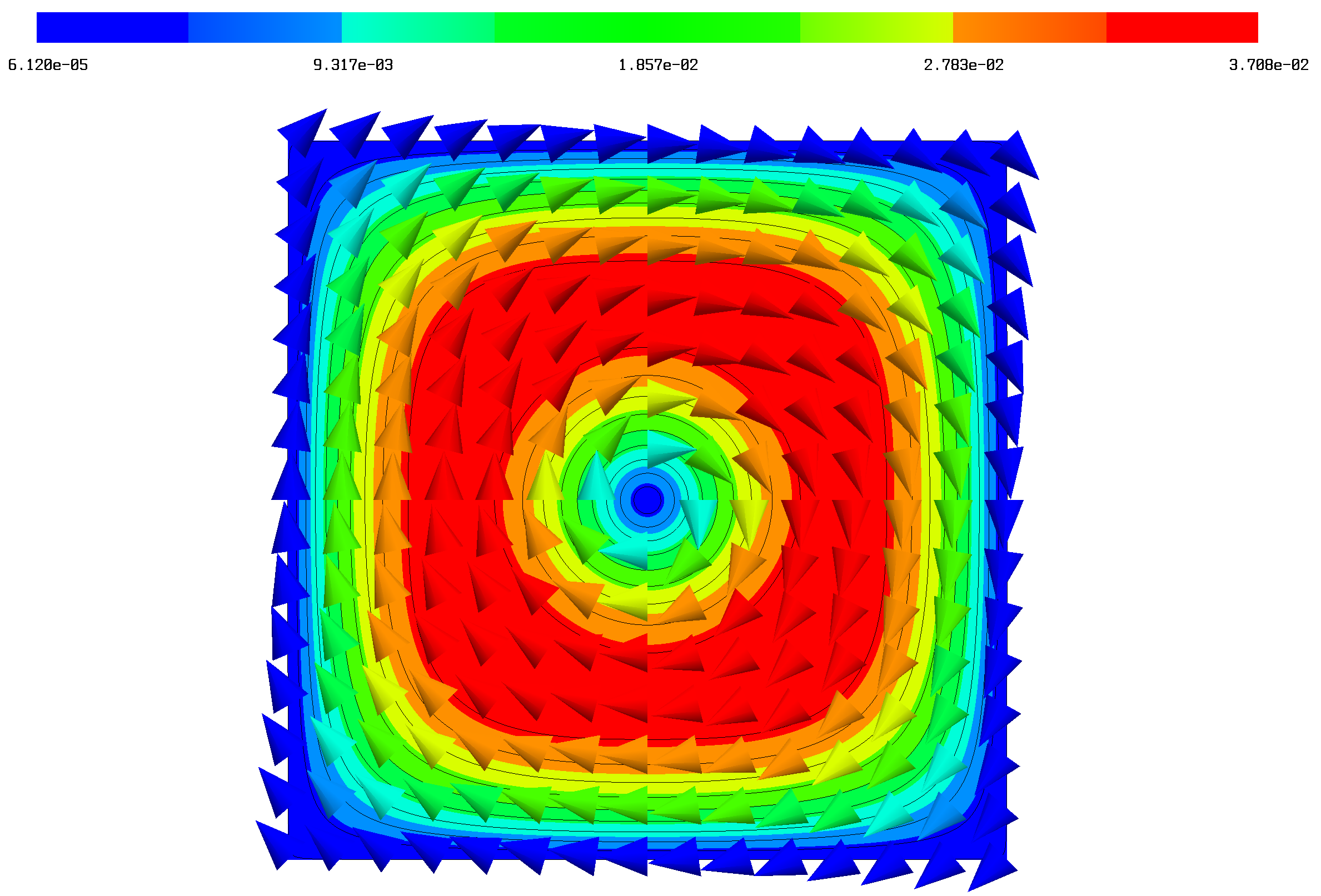}\\
        Ocean current $v_o$ (b)
    \end{minipage}
    \caption{Velocity and ocean current}
    \label{fig:velo}
\end{figure}
To compute our solution, we used the Gauss-Newton method from section \ref{sec:gauss_newton}
with a tolerance of $10^{-4}$.
The needed iterations per time step are shown in Figure \ref{fig:newton}. After about
4 hours only 1 step is needed since the problem arrives at equilibrium. The
experiments show a robust behaviour of Gauss-Newton iterations for decreasing
mesh size.
\begin{figure}[ht]
    \centering

    \begin{tikzpicture}[scale=1]
        \begin{axis}[
                title={Gauss-Newton steps},
                xlabel={Time in s},
                ylabel={Gauss-Newton iteration},
            ]
            \addplot[color=blue, mark=+, mark options=solid] table[x index=0, y index=1, col sep=comma] {data/newton625.txt};
            \addlegendentry{$6.25$ km};
            \addplot[color=blue, mark=triangle, dashdotted , mark options=solid] table[x index=0, y index=1, col sep=comma] {data/newton125.txt};
            \addlegendentry{$12.5$ km};
            \addplot[color=red, mark=o, dashed, mark options=solid] table[x index=0, y index=1, col sep=comma] {data/newton25.txt};
            \addlegendentry{$25$ km};
            \addplot[color=red, mark=x, densely dotted, mark options=solid] table[x index=0, y index=1, col sep=comma] {data/newton50.txt};
            \addlegendentry{$50$ km};           
        \end{axis}
    \end{tikzpicture}
    \caption{Number of Gauss-Newton iterations for given mesh-size to reach tolerance of $10^{-4}$.}
    \label{fig:newton}
\end{figure}
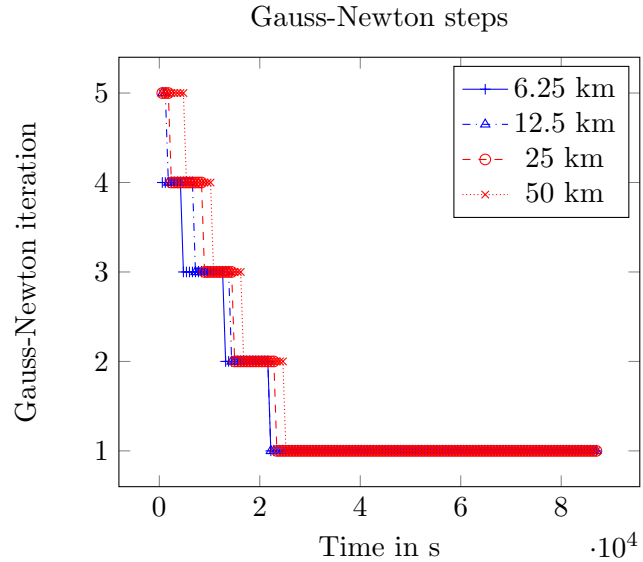
Figure \ref{fig:ls} shows the convergence for the Least-Squares functional for the first
time step and after 1 hour. We see an optimal convergence rate of $2$ for uniform mesh refinement.
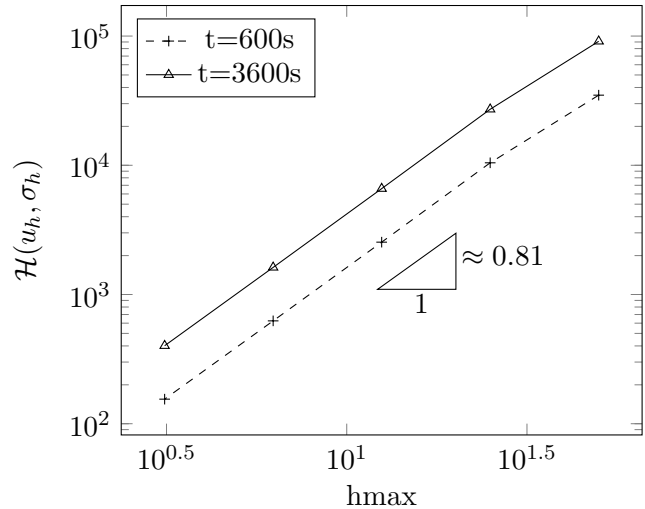
\begin{figure}[ht]
    \centering

    \begin{tikzpicture}[scale=1]
        \begin{loglogaxis}[
                xlabel={hmax},
                ylabel={$\mathcal{H}(u_h, \sigma_h)$},
                legend pos= north west,
            ]
            \addplot[color=black, mark=+, dashed, mark options=solid] table[x index=0, y index=1, col sep=comma] {data/uniform_600s.txt};
            \addlegendentry{t=600s};
            \addplot[color=black, mark=triangle, mark options=solid] table[x index=0, y index=1, col sep=comma] {data/uniform_3600s.txt};
            \addlegendentry{t=3600s};
        \draw (3,8) -- (3,7) -- (2.5,7) --cycle; 
        \node at (3.3,7.6) {$\approx 0.81$}; 
        \node at (2.78,6.75) {1}; 
        \end{loglogaxis}
    \end{tikzpicture}

    \caption{Least-squares functional convergence for a fixed time step and decreasing mesh-size in space}
    \label{fig:ls}
\end{figure}

\subsection{Antarctica}

The second example is the sea-ice motion around Antarctica. We embed the continent 
in a square ocean, such that the ocean has an area of around 32 million square kilometres.
We set Dirichlet boundary conditions at the inner boundary (the land) and on the outer our 
Neumann boundary condition (open ocean). For the sea-ice height and concentration we 
assume the following data constant over time
\begin{align*}
    &h \equiv 1 \mathrm{m}, \\ 
    & A = \sin \left(\pi \frac{x- 3300\mathrm{km}}{6600\mathrm{km}}\right)
                            \left(\pi \frac{y- 3300\mathrm{km}}{6600\mathrm{km}}\right).
\end{align*}
The water velocity is chosen as a cyclone at the upper right part over the domain (Figure \ref{fig:vo_antarcity}). The initial sea-ice 
velocity is zero everywhere and we are again using a stepsize of $600$s.
\begin{figure}[ht]
    \centering
    \includegraphics[scale=0.095]{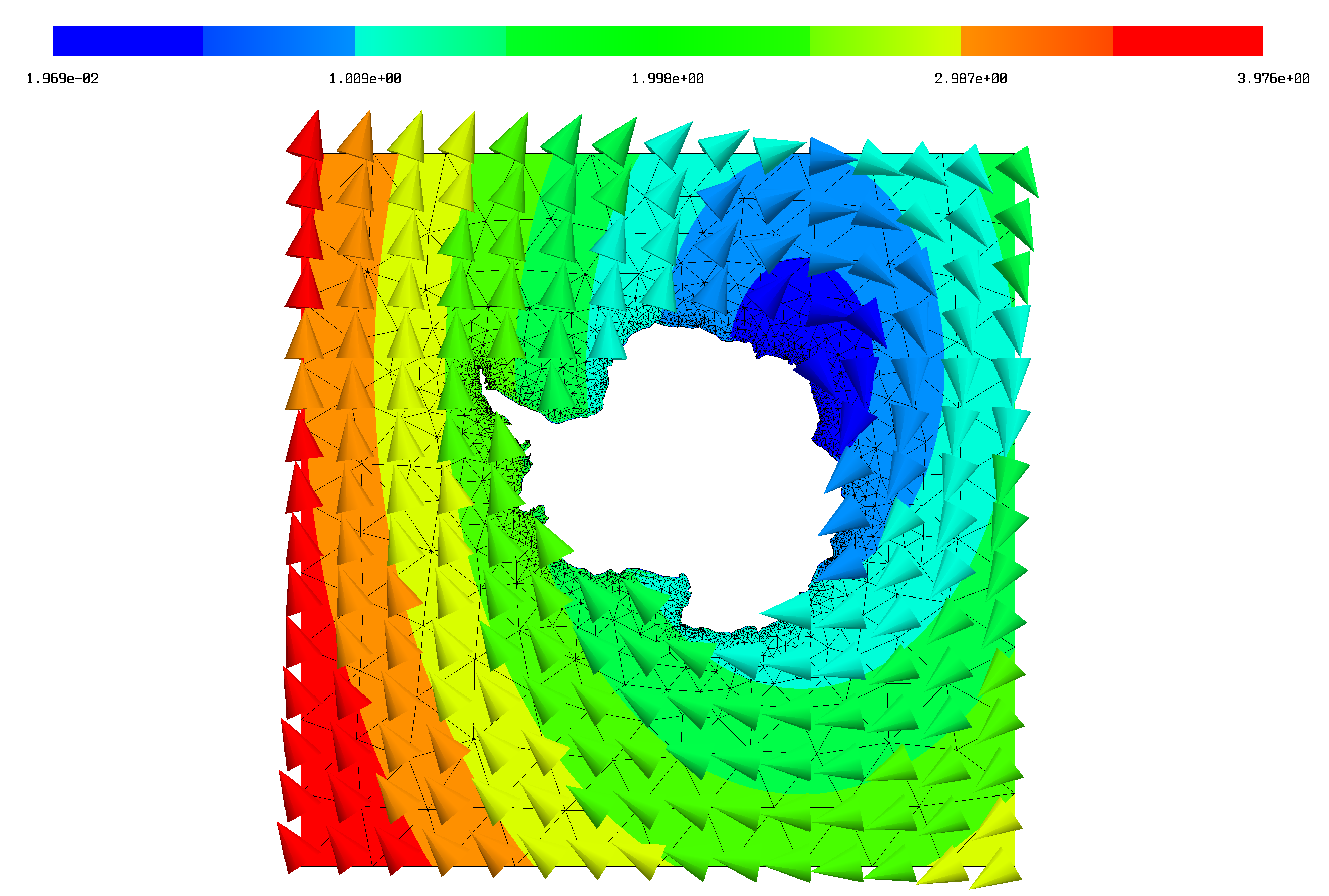}
    \caption{water velocity $v_o$}
    \label{fig:vo_antarcity}
\end{figure}
Since there is no exact solution available we study the 
convergence of the least-squares functional. 
Figure \ref{fig:ls_artarcita_conergence} shows the convergence 
of the least-squares functional for the first time step and after 1 hour.
We observe that the convergence rate and value of the functional are almost 
equal. Like in the first experiment, the convergence rate seems to be limited 
by the regularity of the solution and is approximately 1, where 2 would be optimal.

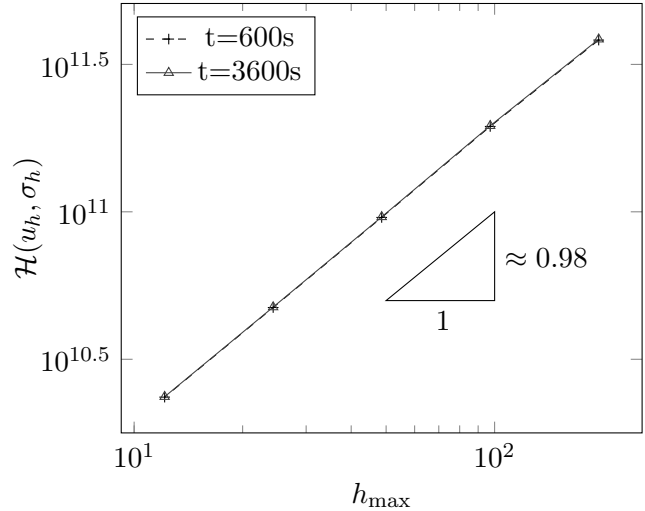
\begin{figure}[ht]
    \centering

    \begin{tikzpicture}[scale=1]
        \begin{loglogaxis}[
                xlabel={$h_{\max}$},
                ylabel={$\mathcal{H}(u_h, \sigma_h)$},
                legend pos= north west,
            ]
            \addplot[color=black, mark=+, dashed, mark options=solid] table[x index=0, y index=1, col sep=comma] {data/a600s.txt};
            \addlegendentry{t=600s};
            \addplot[color=darkgray, mark=triangle, mark options=solid] table[x index=0, y index=1, col sep=comma] {data/a3600s.txt};
            \addlegendentry{t=3600s};
        \draw (axis cs:100 ,1e11)--(axis cs: 100,50000000000)--(axis cs:50,50000000000)--cycle ;
        \node at (axis cs: 72,50000000000) [below] {1};
        \node at (axis cs:100,70000000000) [right] {$\approx 0.98$};
        \end{loglogaxis}
    \end{tikzpicture}

    \caption{Least-squares functional convergence for a fixed time step and decreasing mesh-size in space on the Antarctica domain}
    \label{fig:ls_artarcita_conergence}
\end{figure}
Figure \ref{fig:1h_antarcity} and \ref{fig:24h_antarcity} are showing the sea-ice velocity after 
1 hour (resp. 24 hours). In contrast to the first experiment, we see a clear difference here between 
the sea-water and the sea-ice velocity. Especially some fastened ice is building up around the land mass.
\begin{figure}[ht]
    \centering
    \includegraphics[scale=0.095]{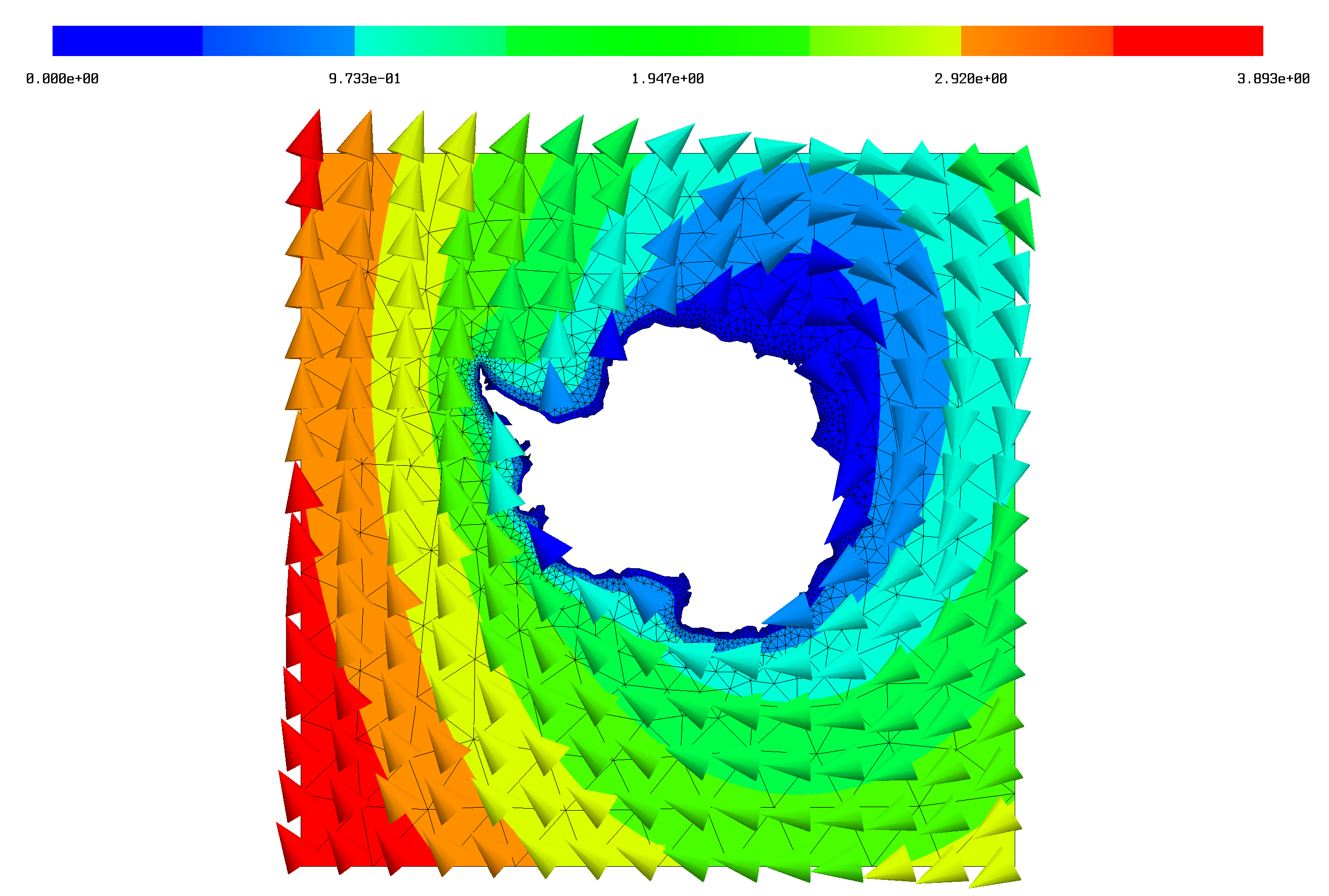}
    \caption{sea-ice velocity $u$ after 1h}
    \label{fig:1h_antarcity}
\end{figure}
\begin{figure}[ht]
    \centering
    \includegraphics[scale=0.095]{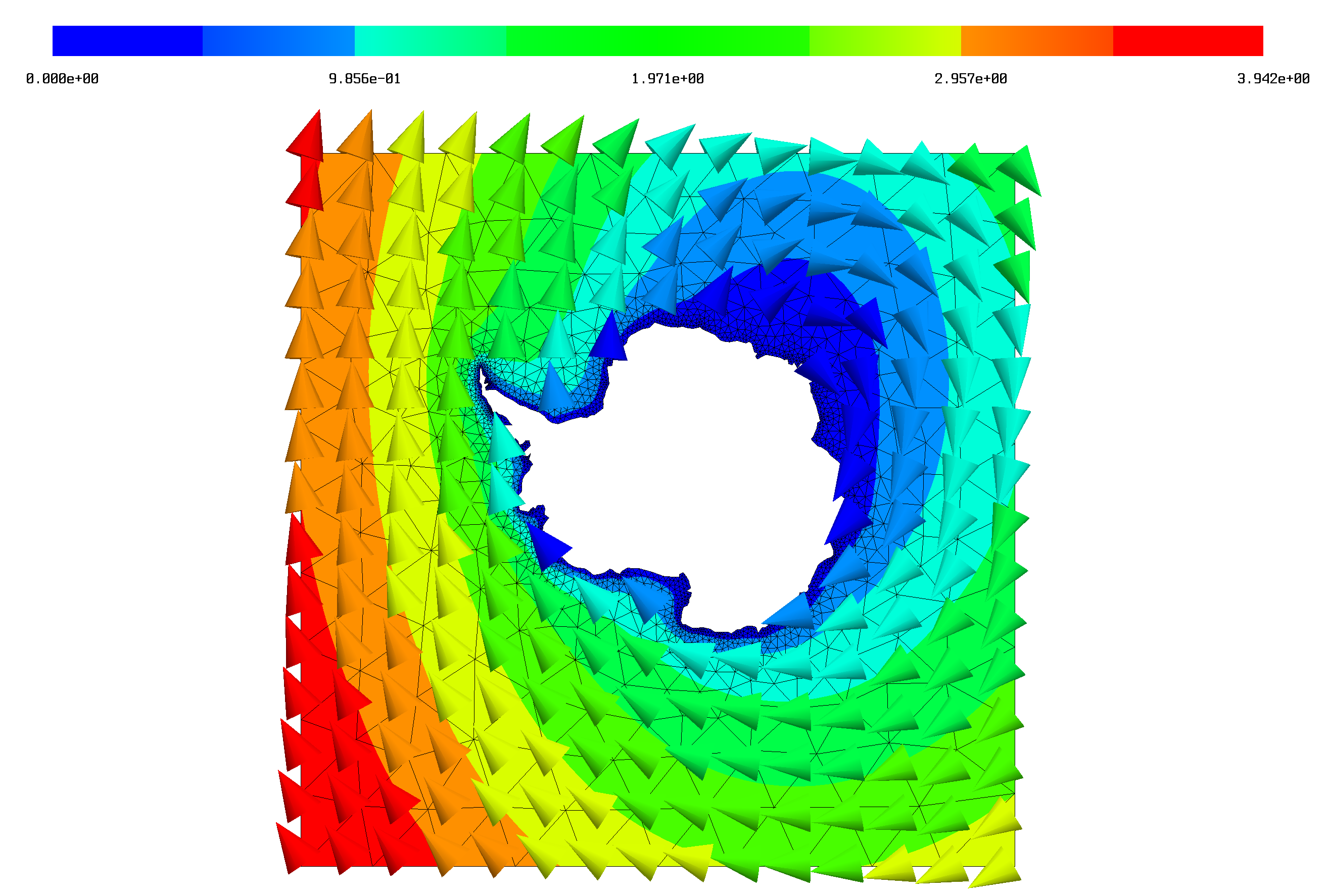}
    \caption{sea-ice velocity $u$ after 24h}
    \label{fig:24h_antarcity}
\end{figure}

Again examine the necessary iteration steps to reach the given 
tolerance of $10^{-4}$. The results are displayed in Figure \ref{fig:newton_antarcia}
and show similar results as before. The needed steps are decreasing with 
decreasing mesh size and the convergence to the stationary state. Overall 
only a small number of steps is necessary, at most 8.

\begin{figure}[ht]
    \centering

    \begin{tikzpicture}[scale=1]
        \begin{axis}[
                xlabel={Time in s},
                ylabel={Gauss-Newton iteration},
            ]
            \addplot[color=blue, mark=+, mark options=solid] table[x index=0, y index=1, col sep=comma] {data/artics_l4.txt};
            \addlegendentry{$12$ km}; 
            \addplot[color=blue, mark=triangle, dashdotted , mark options=solid] table[x index=0, y index=1, col sep=comma] {data/artics_l3.txt};
            \addlegendentry{$24$ km};
            \addplot[color=red, mark=o, dashed, mark options=solid] table[x index=0, y index=1, col sep=comma] {data/artics_l2.txt};
            \addlegendentry{$48$ km};
            \addplot[color=red, mark=x, densely dotted, mark options=solid] table[x index=0, y index=1, col sep=comma] {data/artics_l1.txt};
            \addlegendentry{$97$ km};
            \addplot[color=black, mark=x, densely dotted, mark options=solid] table[x index=0, y index=1, col sep=comma] {data/artics_l0.txt};
            \addlegendentry{$194$ km};
        \end{axis}
    \end{tikzpicture}
    \caption{Number of Gauss-Newton iterations for given mesh-size to reach tolerance of $10^{-4}$.}
    \label{fig:newton_antarcia}
\end{figure}
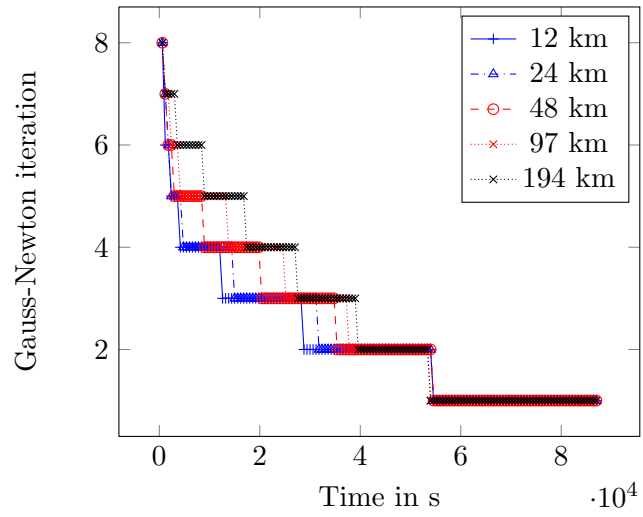
Figure \ref{fig:LS_Antarctica} shows the locally evaluated least-squares functional. 
The image is indicating that the biggest contribution is made on the left coast and 
the open sea is almost nothing contributing to the overall error. This observation can 
be in the future utilised to derive an adaptive mesh refinement algorithm, to make for efficient 
use of computing power.
\begin{figure}[ht]
    \centering
\pgfdeclarelayer{background layer}
\pgfdeclarelayer{foreground layer}
\pgfsetlayers{background layer,main,foreground layer}
\begin{tikzpicture}
    [line cap=round,line join=round,x=1cm,y=1cm,
     spy using outlines={rectangle,lens={scale=4}, 
     connect spies
     },
     decoration={brace,amplitude=2pt}
     ]
  \begin{pgfonlayer}{background layer}
\node[inner sep=0pt, anchor=south west] (S) at (0,0)
    {\includegraphics[width=8cm]{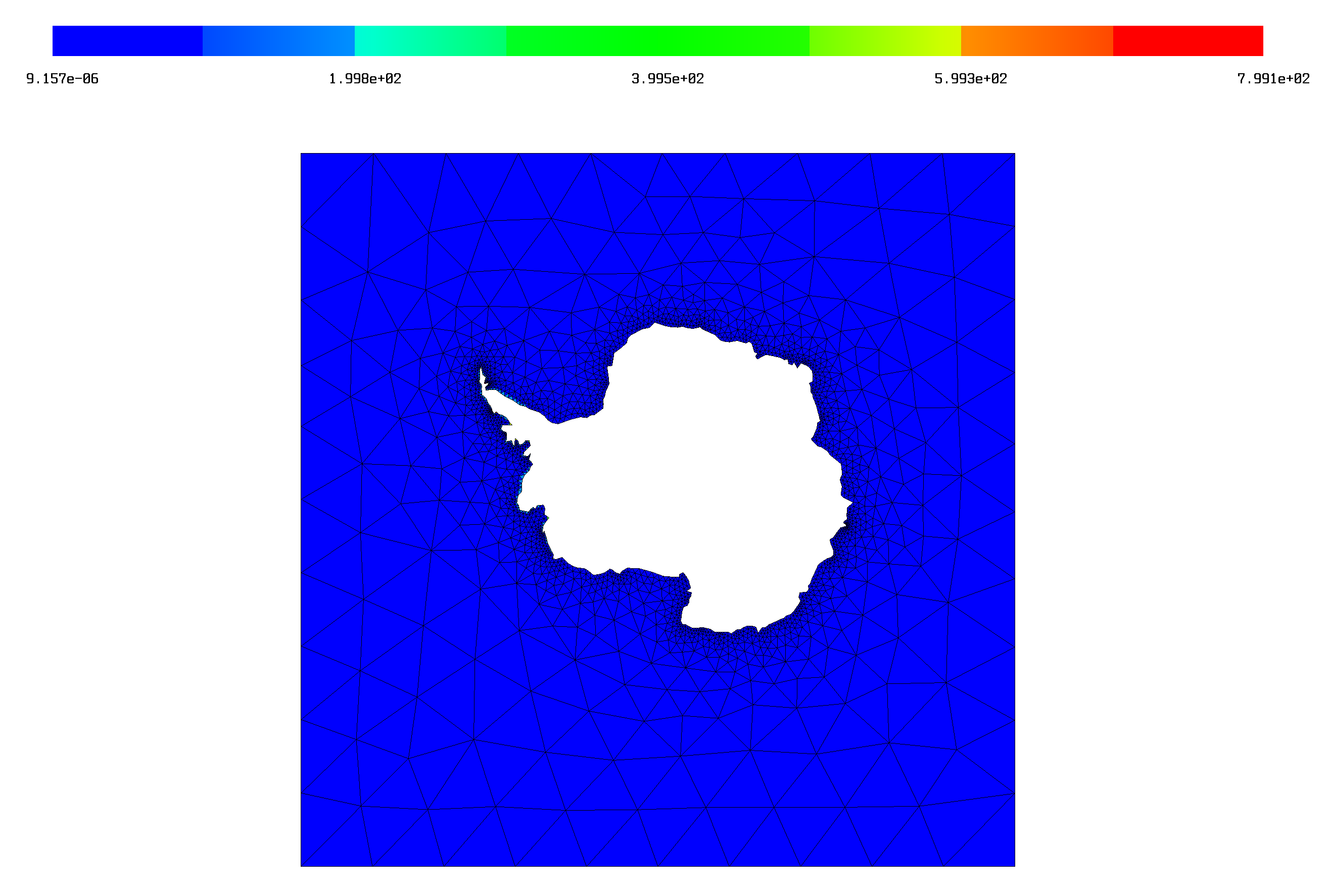}};

   \end{pgfonlayer} 
     \begin{pgfonlayer}{background layer}
\spy [yellow,width=23.7mm, height=16mm,] on (3.0,3.0)
             in node [left] at (3.8,-1);
            \end{pgfonlayer} 
               \begin{pgfonlayer}{background layer}
\spy [green,width=23.7mm, height=16mm,] on (4.27,1.75)
             in node [left] at (6.5,-1);   
                \end{pgfonlayer} 
    \begin{pgfonlayer}{foreground layer}
 \node[inner sep=0pt, anchor=south west] (S) at (1.44,-1.79)
     {\includegraphics[scale=0.0251]{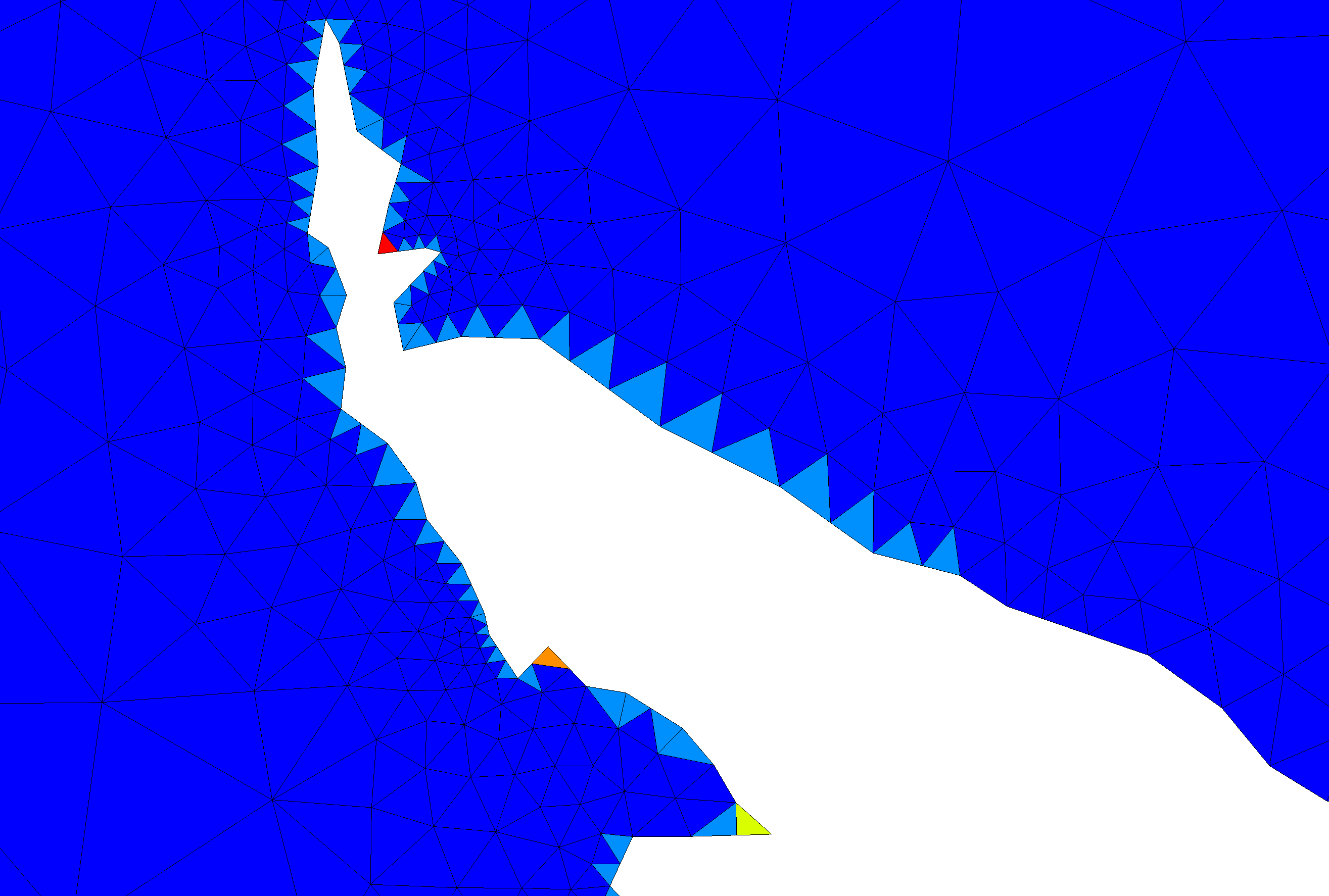}};   
        \end{pgfonlayer} 
 \begin{pgfonlayer}{foreground layer}
 \node[inner sep=0pt, anchor=south west] (S) at (4.14,-1.79)
     {\includegraphics[scale=0.0251]{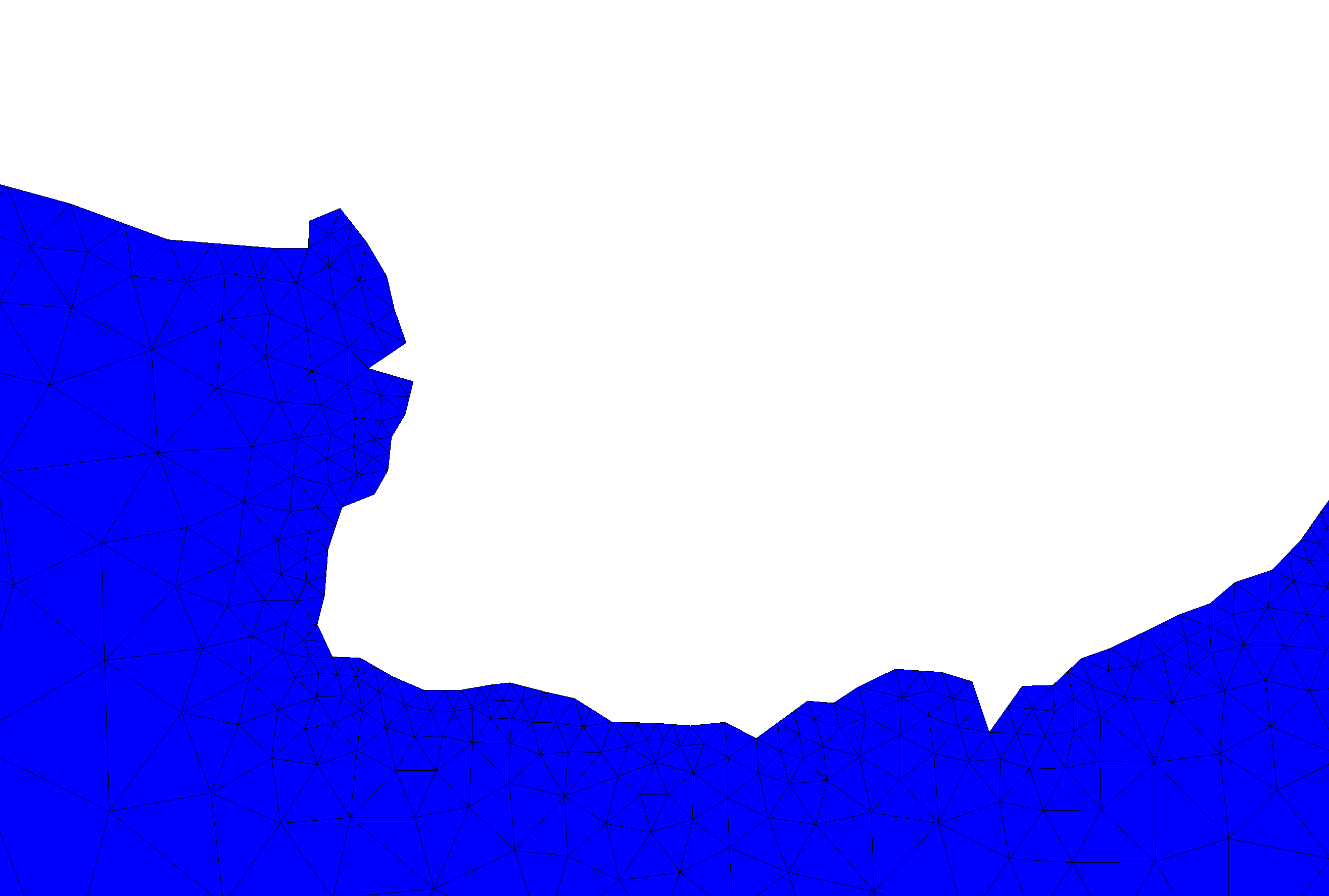}};   
        \end{pgfonlayer} 
\end{tikzpicture}
    \caption{Locally evaluate least-squares functional}
    \label{fig:LS_Antarctica}
\end{figure}

\section*{Acknowledgement}
The second author gratefully acknowledges support by DFG in the Priority Programme SPP 1962 \textit{Non-smooth and Complementarity-based
Distributed Parameter Systems} under grant number STA 402/13-2.

\bibliographystyle{elsarticle-num}
\bibliography{literatur.bib}
\end{document}